\documentclass[12pt]{amsart}
\usepackage{amssymb,bbold,mathrsfs}
\usepackage[all]{xy}

\swapnumbers

\theoremstyle{plain}
\newtheorem{theorem}{Theorem}[section]
\newtheorem{corollary}[theorem]{Corollary}
\newtheorem{lemma}[theorem]{Lemma}
\newtheorem{proposition}[theorem]{Proposition}

\theoremstyle{definition}

\newtheorem{remark}[theorem]{Remark}

\newtheorem{numbered}[theorem]{}
\newtheorem{example}[theorem]{Example}

\theoremstyle{remark}

\newcommand{\abs}[1]{\lvert#1\rvert}
\newcommand{\norm}[1]{\lVert#1\rVert}

\newcommand{\bignorm}[1]{\bigl\lVert#1\bigr\rVert}

\renewcommand{\le}{\leqslant}
\renewcommand{\ge}{\geqslant}

\newcommand{\term}[1]{{\textit{\textbf{#1}}}}

\renewcommand{\mid}{\::\:}

\def\one{\mathbb 1}

\def\iA{\mathcal A}
\def\iJ{\mathcal J}
\def\iS{\mathscr{S}}
\def\iP{\mathscr{P}}
\def\iG{\mathcal G}
\def\Rplus{\overline{\mathbb R^+}}

\DeclareMathOperator{\Range}{Range}
\DeclareMathOperator{\Span}{span}

\DeclareMathOperator{\rank}{rank}
\DeclareMathOperator{\minrank}{min\,rank}

\makeindex

\begin{document}
\baselineskip 18pt

\title[Irreducible semigroups]{Irreducible semigroups of positive
  operators on Banach lattices}

\author[N.Gao]{Niushan Gao}
\author[V.G.Troitsky]{Vladimir G. Troitsky}
\address{Department of Mathematical
  and Statistical Sciences, University of Alberta, Edmonton,
  AB, T6G\,2G1. Canada}
\email{niushan@ualberta.ca, troitsky@ualberta.ca}
\thanks{The second author was supported by NSERC}
\keywords{irreducible semigroup, positive operator, invariant ideal,
  Riesz operator}
\subjclass[2010]{Primary: 47D03. Secondary: 47B65, 47A15, 47B06.}

\date{\today}

\begin{abstract}
  The classical Perron-Frobenius theory asserts that an irreducible
  matrix $A$ has cyclic peripheral spectrum and its spectral radius
  $r(A)$ is an eigenvalue corresponding to a positive eigenvector. In
  \cite{Radjavi:99,Radjavi:00}, this was extended to semigroups of
  matrices and of compact operators on $L_p$-spaces. We extend this
  approach to operators on an arbitrary Banach lattice $X$. We prove,
  in particular, that if $\iS$ is a commutative irreducible semigroup
  of positive operators on $X$ containing a compact operator $T$ then
  there exist positive disjoint vectors $x_1,\dots,x_r$ in $X$ such
  that every operator in $\iS$ acts as a positive scalar multiple of a
  permutation on $x_1,\dots,x_r$. Compactness of $T$ may be replaced
  with the assumption that $T$ is peripherally Riesz, i.e., the
  peripheral spectrum of $T$ is separated from the rest of the
  spectrum and the corresponding spectral subspace $X_1$ is finite
  dimensional. Applying the results to the semigroup generated an
  irreducible peripherally Riesz operator $T$, we show that $T$ is a
  cyclic permutation on $x_1,\dots,x_r$, $X_1=\Span\{x_1,\dots,x_r\}$,
  and if $S=\lim_j b_jT^{n_j}$ for some $(b_j)$ in $\mathbb R_+$ and
  $n_j\to\infty$ then $S=c(T_{|X_1})^k\oplus 0$ for some $c\ge 0$ and
  $0\le k<r$. We also extend results of
  \cite{Abramovich:92a,Grobler:95} about peripheral spectra of
  irreducible operators.
\end{abstract}

\maketitle

\section{Introduction}

Recall that a square matrix $A$ with non-negative entries is said to
be irreducible if no permutation of the basis vectors brings it to a
block form
\begin{math}
  \bigl[
  \begin{smallmatrix}
    A_{11} & A_{12}\\
    0      & A_{22}
  \end{smallmatrix}
  \bigr].
\end{math}
The classical Perron-Frobenius theory (see, e.g.,
\cite[Theorem~8.26]{Abramovich:02}) asserts that for such a matrix,
its spectral radius $r(A)$ is non-zero, its peripheral eigenvalues
(the ones whose absolute value is $r(A)$) are exactly the $m$-th roots
of unity for some $m\in\mathbb N$, and the corresponding eigenspaces
are one dimensional. Moreover, the eigenspace for $r(A)$ itself is
spanned by a vector whose coordinates are all positive.

There have been numerous extensions and generalizations of
Perron-Frobenius Theory. In particular, instead of a positive matrix,
one can consider a positive operator on a Banach lattice, or even a
family of positive operators. We say that such a family is \term{ideal
  irreducible} if it has no common invariant closed non-zero proper
ideals; it is \term{band irreducible} if it has no common invariant
proper non-zero bands. In particular, a positive operator is ideal
irreducible (band irreducible) if it has no invariant proper non-zero
closed ideals (respectively, bands). It is easy to see that in case of
a single positive matrix, these definitions coincide with
irreducibility. We refer the reader to~\cite{Abramovich:02} for
details and terminology on Banach lattices and irreducible operators.

There have been many extensions of Perron-Frobenius Theory to ideal or
band irreducible operators on Banach lattices; see, e.g.,
\cite{Niiro:66,Schaefer:74,dePagter:86,Grobler:95,Abramovich:02,Kitover:05}
etc., and references there.
In most of these extensions, it is assumed that the operator is
compact, or \term{power compact} (i.e., some power of it is compact),
or, at least, the spectral radius is a pole of the resolvent. There
have also been some extensions to semigroups of positive operators.
For example, Drnov\v sek in~\cite{Drnovsek:01} proved that an ideal
irreducible semigroup of compact positive operators must contain a
non-quasi\-nilpotent operator.

In \cite{Radjavi:99} and in Sections~5.2 and~8.7 of~\cite{Radjavi:00},
a different approach was used to extend Perron-Frobenius Theory from a
single irreducible matrix to an irreducible semigroups of matrices or
of compact operators on $L_p(\mu)$ ($1\le
p<+\infty$). In~\cite{Levin:09}, this approach was applied to order
continuous Banach lattices. In the current paper, we extend it to
arbitrary Banach lattices. Some of the ideas we use are parallel to
those used in~\cite{Radjavi:99,Radjavi:00}, but in many cases we had
to develop completely new techniques.  Some of our results are new
even in the case of $L_p(\mu)$ and in the single operator
case. Moreover, we weaken the condition that the semigroup consists
entirely of compact operators; we only require that the semigroup
contains a compact or even a peripherally Riesz operator. An operator
$T$ is said to be \term{peripherally Riesz} if its \term{peripheral
  spectrum}
\begin{math}
  \sigma_{\rm per}(T)=
  \bigl\{\lambda\in\sigma(T)\mid\abs{\lambda}=r(T)\bigr\}
\end{math}
consists of isolated eigenvalues of finite multiplicity.  This class
contains all non-quasi\-nilpotent compact and strictly singular
operators.

The paper is structured as follows. Recall that a set of positive
operators is $\Rplus$\term{-closed} if is norm closed and it is closed
under multiplication by positive scalars. Since taking the
$\Rplus$-closure of a semigroup does not affect its invariant closed
ideals, we may assume without loss of generality that our semigroups
are $\Rplus$-closed. In Section~\ref{sec:pR} we show that the
$\Rplus$-closed semigroup generated by a peripherally Riesz operator
either contains the peripheral spectral projection of the operator or
a non-zero nilpotent operator of small finite rank. We use this in
Section~\ref{sec:fin-rank} to show that if $\iS$ is an $\Rplus$-closed
ideal irreducible semigroup and $\iS$ contains a peripherally Riesz or
a compact operator, then it contains operators of finite
rank. Moreover, it contains ``sufficiently many'' projections of rank
$r$, where $r$ is the minimal non-zero rank of operators in $\iS$.  In
Section~\ref{sec:sr} we discuss the special case when all such
projections have the same range (this is the case when $\iS$ is
commutative, in particular, when $\iS$ is generated by a single
operator). We show that, in this case, there are disjoint
vectors $x_1,\dots,x_r$ in $X_+$ such that each operator in the
semigroup acts on these vectors as a scalar multiple of a
permutation. In particular, $x_0:=x_1+\dots+x_r$ is a common
eigenvector for $\iS$. In Section~\ref{sec:up} we show that the dual
semigroup $\{S^*\mid S\in\iS\}$ has the same properties under the
somewhat stronger condition that $\iS$ has a unique projection of rank
$r$ (which is still satisfied when $\iS$ is commutative). In
Section~\ref{sec:app}, we apply our results to finitely generated
semigroups. We completely characterize $\iS$ in the case when it is
generated by a single peripherally Riesz ideal irreducible operator
$T$; we show that $T$ acts as a scalar multiple of a cyclic
permutation of $x_1,\dots,x_r$. We improve
\cite[Corollary~9.21]{Abramovich:02} that if $S$ and $K$ are two
positive commuting operators such that $K$ is compact and $S$ is ideal
irreducible then $r(K)>0$ and $r(S)>0$; we show that in this case
$\lim_n\norm{K^nx}^{\frac{1}{n}}=r(K)$ and
$\liminf_n\norm{S^nx}^{\frac{1}{n}}>0$ whenever $x>0$. In
Section~\ref{sec:bi}, we extend the results of the preceding sections
to band irreducible semigroups of order continuous operators. In
particular, it allows us to improve Grobler's characterization of the
peripheral spectrum of a band irreducible power compact operator
in~\cite{Grobler:95}. Finally, in Section~\ref{sec:ideals} we
investigate the structure of one-sided ideals in an irreducible
semigroup.

\section{Peripherally Riesz operators}
\label{sec:pR}

Given a set $\iA$ of operators on a Banach space $X$, we write
$\Rplus\iA$ for the smallest $\Rplus$-closed semigroup containing
$\iA$.  In particular, if $T$ is an operator on $X$, we will write
$\Rplus T$ for the $\Rplus$-closed semigroup generated by
$T$. Clearly, $\Rplus T$ consists of all positive scalar multiples of
powers of $T$ and of all the operators of form $\lim_jb_jT^{n_j}$
for some sequence $(b_j)$ in $\mathbb R_+$ and some strictly
increasing sequence $(n_j)$ in $\mathbb N$; these limit operators form
the \term{asymptotic} part of $\Rplus T$.

Given a semigroup $\iS$ in $L(X)$, we will denote by $\minrank\iS$ the
minimal rank of non-zero elements of $\iS$; we write
$\minrank\iS=+\infty$ if $\iS$ contains no non-zero operators of
finite rank. Note that if $T\in\iS$ then the ideal generated by $T$
in $\iS$ consists of all the operators of form $ATB$ where
$A,B\in\iS\cup\{I\}$.

A vector $u\in\mathbb C^n$ is said \term{unimodular} if $\abs{u_i}=1$
for all $i=1,\dots,n$. Let $\mathbb U_n$ denote the set of all
unimodular vectors in $\mathbb C^n$. Clearly $\mathbb U_n$ is a group
with respect to the coordinate-wise product, with unit $\one=(1,\dots,1)$.
We will need the following standard lemma.

\begin{lemma}\label{Kron}
  If $u\in\mathbb U_n$ then there exists a strictly increasing sequence
  $(m_j)$ in $\mathbb N$ such that $u^{m_j}\to\one$.
\end{lemma}

\begin{proof}
  Since $\mathbb U_n$ is compact, we can find a subsequence
  $u^{k_j}\to v$ for some $v\in\mathbb U_n$. Passing to a subsequence,
  we may assume that $m_j:=k_{j+1}-k_j$ is strictly increasing. Then
  $u^{m_j}=u^{k_{j+1}}u^{-k_j}\to vv^{-1}=\one$.
\end{proof}

A square matrix $A$ is \term{unimodular} if there is a basis in which
it is diagonal and the diagonal is a unimodular vector. It follows
from Lemma~\ref{Kron} that in this case $(A^{m_j})$ converges to the
identity matrix.

\begin{numbered}\label{matrix-dych}
  The following observation is based on Lemma~1 of~\cite{Radjavi:99}
  and is critical for our study. Let $A$ be a square matrix with
  $r(A)=1$ and $\sigma(A)=\sigma_{\rm per}(A)$. Using Jordan
  decomposition of $A$, we can write $A=U+N$ where $U$ is unimodular,
  $N$ is nilpotent, and $UN=NU$. By Lemma~\ref{Kron}, we can find a
  strictly increasing sequence $(m_j)$ such that $U^{m_j}\to I$.

  \emph{Case $N=0$.} In this case, $A=U$, so that
  $A^{m_j}\to I$.

  \emph{Case $N\ne 0$.} Let $k$ be such that $N^k\ne 0$ but
  $N^{k+1}=0$. Then
  \begin{equation}\label{binom}
    A^n=(U+N)^n=
    U^n+\binom{n}{1}U^{n-1}N+\dots+\binom{n}{k}U^{n-k}N^k.
  \end{equation}
  Note that $\lim_n\binom{n}{i}/\binom{n}{k}=0$ whenever
  $i<k$. Therefore, if we divide \eqref{binom} by $\binom{n}{k}$, then
  every term in the sum except the last one converges to zero as
  $n\to\infty$. Denote $r_j=m_j+k$ and $c_j=1/\binom{r_j}{k}$, then
  \begin{math}
    \lim_jc_jA^{r_j}=\lim_jU^{r_j-k}N^k=N^k.
  \end{math}
  We can now summarize as follows.
\end{numbered}

\begin{proposition}\label{dych-id-nilp}
  Let $A$ be a square matrix with $r(A)=1$ and $\sigma(A)=\sigma_{\rm
    per}(A)$. Then exactly one of the following holds:
  \begin{enumerate}
  \item $A$ is unimodular and $A^{m_j}\to I$ for some strictly increasing
    sequence $(m_j)$ in $\mathbb N$; or
  \item There exist a strictly increasing sequence $(r_j)$ in $\mathbb N$ and a
    sequence $(c_j)$ in $\mathbb R_+$ such that $c_j\downarrow 0$ and
    $c_j A^{r_j}$ converges to a non-zero nilpotent (even square-zero) matrix.
   \end{enumerate}
\end{proposition}

We will refer to these two cases as ``unimdular'' and
``nilpotent''. In the unimodular case, it is easy to see that every
operator in $\Rplus A$ is a scalar multiple of a unimodular
operator. The following proposition describes the asymptotic part of
$\Rplus A$ in the nilpotent case.

\begin{proposition}\label{matrix-nilp-sg}
  Let $A$ be a square matrix with $r(A)=1$ and $\sigma(A)=\sigma_{\rm
    per}(A)$. Suppose that the nilpotent part $N$ of the Jordan
  decomposition of $A$ is non-zero. If
  $B=\lim_jb_jA^{n_j}$ with $(n_j)$ strictly increasing, then $B$ is
  nilpotent (even square-zero) and $b_j\to 0$.
\end{proposition}

\begin{proof}
  We will use the notations of~\ref{matrix-dych}.
  Recall that the matrix $U$ is unimodular with respect to some basis
  $e_1,\dots,e_n$. For $x=\sum_{i=1}^nx_ie_i$, put
  $\norm{x}=\sum_{i=1}^n\abs{x_i}$. Clearly, this is a norm on
  $\mathbb R^n$ and $U$ is an isometry with respect to this norm.  It
  follows from~\eqref{binom} that $\binom{n}{k}^{-1}A^n-U^{n-k}N^k\to
  0$ as $n\to\infty$.  Since $U$ is an isometry and $N^k\ne 0$, the
  sequence $\bigl(\norm{U^{n-k}N^k}\bigr)_n$, and therefore
  $\bigl(\binom{n}{k}^{-1}\norm{A^n}\bigr)_n$, is bounded above and
  bounded away from zero. It follows from $b_jA^{n_j}\to B$ that the
  sequence $\bigl(b_j\binom{n_j}{k}\bigr)_j$ is bounded, hence $b_j\to
  0$. It also follows that $b_j\binom{n_j}{k}U^{n_j-k}N^k\to B$ so
  that $B^2=\lim_j\Bigl(b_j\binom{n_j}{k}U^{n_j-k}N^k\Bigr)^2=0$
  because $UN=NU$ and $N^{2k}=0$.
\end{proof}

Let $T$ be an operator on a Banach space $X$. Recall that $T$ is said
to be \term{Riesz} if its non-zero spectrum consists of isolated
eigenvalues with finite-dimensional spectral subspaces. Equivalently,
the essential spectral radius $r_{\rm ess}(T)$ is zero. In particular,
compact and strictly singular operators are Riesz.  We will be
interested in the asymptotic part of $\Rplus T$, and it is really only
determined by the restriction of $T$ to its \term{peripheral spectral
  subspace}, i.e., the spectral subspace corresponding to $\sigma_{\rm
  per}(T)$. This motivates the following definition: we say that $T$
is \term{peripherally Riesz} if $r(T)>0$, $\sigma_{\rm per}(T)$ is a
spectral set (i.e., it is separated from the rest of the spectrum),
and the peripheral spectral subspace is finite-dimensional. It is
often convenient to assume, in addition, that $r(T)=1$; this can
always be achieved by scaling $T$. Note that $T$ is peripherally Riesz
iff $r_{\rm ess}(T)<r(T)$; in this case, $\sigma_{\rm per}(T)$
consists of poles of the resolvent. In particular, every
non-quasi\-nilpotent Riesz operator is peripherally Riesz.
Applying the results of the first part of this section, we obtain the
following two possible structures of the asymptotic part of
$\Rplus T$.

\begin{proposition}\label{dych-proj-nilp}
  Suppose that $T$ is peripherally Riesz with $r(T)=1$. Let
  $X=X_1\oplus X_2$, where $X_1$ and $X_2$ are the spectral subspaces
  for $\sigma_{\rm per}(T)$ and its complement, respectively. Let $P$
  be the spectral projection onto $X_1$. Then exactly one of the
  following holds.
  \begin{enumerate}
  \item\label{dpn-unim} \emph{(``Unimodular'' case)} $T_{|X_1}$ is
    unimodular, and each operator in the asymptotic part of $\Rplus T$
    is of form $cU\oplus 0$, where $c\ge 0$ and $U$ is unimodular. Some
    sequence $(T^{m_j})$ of powers of $T$ converges to $P$, $P$ is the
    only non-zero projection in $\Rplus T$, and $\Rplus T$ contains no
    non-zero quasi-nilpotent operators.
  \item\label{dpn-nilp} \emph{(``Nilpotent'' case)} The asymptotic
    part of $\Rplus T$ is non-trivial. For each operator $S$ 
    with $S=\lim_jb_jT^{n_j}$ with $(n_j)$ strictly increasing, we have 
    $S=B\oplus 0$ where $B\in L(X_1)$ is nilpotent (even square-zero)
    and $b_j\to 0$. Also, $\Rplus T$ contains no projections.
  \end{enumerate}
\end{proposition}

\begin{proof}
  Let $T_1=T_{|X_1}$ and $T_2=T_{|X_2}$.

  \eqref{dpn-unim} Suppose that $T_1$ is unimodular. Take $S\in\Rplus
  T$. As $X_1$ and $X_2$ are invariant under $S$, we can write
  $S=S_1\oplus S_2$. Suppose that $S=\lim_jb_jT^{n_j}$ for some
  $(b_j)$ in $\mathbb R_+$ and some strictly increasing sequence
  $(n_j)$ in $\mathbb N$. Then $b_jT_1^{n_j}\to S_1$. It follows that
  $S_1$ is a scalar multiple of a unimodular matrix and $(b_j)$ is
  bounded. It now follows from $r(T_2)<1$ that
  $S_2=\lim_jb_jT_2^{n_j}=0$. So $S$ is of form $cU\oplus
  0$. Furthermore, for every non-zero $S\in\Rplus T$, the restriction
  $S_{|X_1}$ is a positive scalar multiple of a unimodular matrix, so
  that $S$ is not quasinilpotent.

  By Proposition~\ref{dych-id-nilp}, $T_1^{m_j}$ converges to the
  identity of $X_1$. Since $r(T_2)<1$ we have $T_2^{m_j}\to
  0$. Therefore, $T^{m_j}\to P$. Finally, let's show that $P$ is the
  only non-zero projection in $\Rplus T$. Suppose $Q\in\Rplus T$ is a
  projection. Suppose first that $Q=cT^n$ for some $c>0$ and
  $n\in\mathbb N$. Then 
  \begin{math}
    \frac{1}{c^{m_j}}Q=\bigl(\frac{1}{c}Q)^{m_j}=T^{nm_j}\to P^n=P;
  \end{math}
  it follows that $c=1$ and $Q=P$. Suppose now that $Q$ is in the
  asymptotic part of $\Rplus T$.  Then $Q=Q_1\oplus 0$ where $Q_1$ is
  unimodular and is a projection in $L(X_1)$; hence $Q$ is the
  identity on $X_1$ and, therefore, $Q=P$.

  \eqref{dpn-nilp} Suppose now that $T_1$ is not unimodular, hence it
  has a non-trivial nilpotent part. By Proposition~\ref{dych-id-nilp},
  there exist sequences $(c_j)$ in $\mathbb R_+$ and $(r_j)$ in
  $\mathbb N$ such that $c_j\to 0$, $(r_j)$ is strictly increasing,
  and $(c_jT_1^{r_j})$ converges to a non-zero square-zero operator
  $C$ on $X_1$. It follows from $c_j\to 0$ and $r(T_2)<1$ that
  $c_jT_2^{r_j}\to 0$. Therefore, $c_jT^{r_j}\to C\oplus 0$, hence
  $C\oplus 0$ is in the asymptotic part of $\Rplus T$.

  Suppose that $S=\lim_jb_jT^{n_j}$ for some $(b_j)$ in $\mathbb R_+$
  and some strictly increasing
  $(n_j)$. Proposition~\ref{matrix-nilp-sg} applied with $A=T_1$
  guarantees that $b_j\to 0$ and $S_{|X_1}$ is a square-zero
  operator. Furthermore, $r(T_2)<1$ implies
  $S_{|X_2}=\lim_jb_jT_2^{n_j}=0$. In particular, $S$ cannot be a
  projection.

  It is left to show that if $Q=cT^n$ for some $c>0$ and $n\in\mathbb
  N$ then $Q$ is not a projection. Suppose it is. It follows from
  $r(Q)=1=r(T^n)$ that $c=1$, so $Q=T^n$. Hence, the set of all
  distinct powers of $T$ is finite. It follows from $c_j\to 0$ that
  $c_jT^{r_j}\to 0$, but this contradicts $c_jT^{r_j}\to C\oplus 0\ne
  0$.
\end{proof}

\begin{remark}\label{dych-no-nilp}
  Suppose that, in addition, $\rank T=\minrank\Rplus T<\infty$. Then
  the nilpotent case in Proposition~\ref{dych-proj-nilp} is
  impossible. Indeed, otherwise $\Rplus T$ would contain an operator
  of the form $C\oplus 0$ where $C$ is a nilpotent operator in
  $L(X_1)$, hence $$0<\rank C\oplus 0=\rank C<\dim X_1\le\rank T$$
  since $T$ is an isomorphism on $X_1$; a contradiction. Thus, we have
  $P\in \Rplus T$, where $P$ is the spectral projection for $X_1$. It
  follows that $\rank T=\rank P=\dim X_1$, so that
  $T_{|X_2}=0$. Hence, $\Range T=X_1$, $\ker T=X_2$, and $\sigma(T)$
  consists of $\sigma_{\rm per}(T)$ and, possibly, zero.
\end{remark}

\section{$\Rplus$-closed semigroups on Banach spaces}

Throughout this section, we assume that $\iS$ is an $\Rplus$-closed
semigroup of operators on a Banach space $X$. The following result
follows immediately from Proposition~\ref{dych-proj-nilp}.

\begin{proposition}\label{pR-finrank}
  If $\iS$ contains a peripherally Riesz operator then $\iS$ contains
  a finite-rank operator.
\end{proposition}

In particular, this proposition applies when $\iS$ contains a
non-quasi\-nilpotent compact or even strictly singular operator.

Can we find not just a finite-rank operator in $\iS$ but a finite-rank
projection? As in Remark~\ref{dych-no-nilp}, if there is a $T\in\iS$
such that $\rank T=\minrank\iS<+\infty$ and $T$ is not nilpotent then
the spectral projection $P$ for $\sigma_{\rm per}(T)$ is in $\iS$ and
$\rank P=\rank T$. The next lemma shows that in this case $\iS$
contains ``sufficiently many'' projections.

\begin{lemma}\label{rank-r-proj}
  Suppose that $S\in\iS$ such that $r:=\rank S=\minrank\iS<\infty$ and
  $S$ is not nilpotent. Then there exist projections $P$ and $Q$ in
  $\iS$ with $\rank P=\rank Q=r$ and $PS=SQ=S$. Moreover, the
  condition ``$S$ is not nilpotent'' may be replaced with ``$AS$ is
  not nilpotent for some $A\in\mathcal S$''.
\end{lemma}

\begin{proof}
  Suppose $AS$ is not nilpotent for some $A\in\mathcal S$ or $A=I$. Then
  $r(SA)=r(AS)\ne 0$. Clearly, $\rank AS=\rank SA=r$. It
  follows from $\Range SA\subseteq\Range S$ and $\rank SA=\rank S$
  that $\Range SA=\Range S$. By the preceding remark with $T=SA$, the
  peripheral spectral projection $P$ of $SA$ is in $\iS$, $\rank P=r$,
  and $\Range P=\Range SA=\Range S$, hence $PS=S$.

  In order to find $Q$, we pass to the adjoint semigroup
  $\iS^*=\{T^*\mid T\in\iS\}$. Note that $\iS^*$, $S^*$, and $A^*$
  still satisfy all the assumptions of the lemma, so we can find a
  projection $R\in\iS^*$ such that $\rank R=r$ and $RS^*=S^*$. Then
  $R=Q^*$ for some projection $Q\in\iS$ with $\rank Q=r$ and $SQ=S$.
\end{proof}

\begin{lemma}\label{matrix-group}
  Suppose that $\iS$ is an $\Rplus$-closed semigroup of matrices such
  that every non-zero matrix in $\iS$ is invertible. Then
  $\bigl\{A\in\iS\mid r(A)=1\bigr\}$ is a closed group.
\end{lemma}

\begin{proof}
  Let $\iS_1:=\bigl\{A\in\iS\mid r(A)=1\bigr\}$. Take any
  $A\in\iS_1$. Since $\iS$ contains no non-zero nilpotent matrices,
  the nilpotent case in Proposition~\ref{dych-proj-nilp} is impossible,
  hence some sequence of powers $A^{m_j}$ converges to the peripheral
  spectral projection $P$ of $A$. In particular, $P\in\iS$, hence
  invertible, so that $P=I$ and $\sigma(A)$ is contained in the unit
  circle. This yields that $A$ is unimodular. It follows from
  $A^{m_j-1}=A^{-1}A^{m_j}\to A^{-1}$ that $A^{-1}\in\iS$. Clearly,
  $\sigma(A^{-1})$ is also contained in the unit circle, so
  that $A^{-1}\in\iS_1$.

  Suppose that $0\ne A\in\iS$. Then $\frac{1}{r(A)}A\in\iS_1$, and the
  later matrix is unimodular, so that $\abs{\det A}=r(A)^n$.  It
  follows that for $A\in\iS$ we have $A\in\iS_1$ iff $\abs{\det A}=1$.
  Therefore, $\iS_1$ is closed under multiplication. It also follows
  that $\iS_1$ is closed.
\end{proof}

\section{Ideal irreducible semigroups containing finite-rank operators.}
\label{sec:fin-rank} 

Throughout this section, $\iS$ is a semigroup of positive operators on
a Banach lattice $X$. For $x\in X$, the \term{orbit} of $x$ under
$\iS$ is defined as $\iS x=\{Sx\mid S\in\iS\}$. We will use the
following known fact; cf. Lemma~8.7.6 in~\cite{Radjavi:00} and
Proposition~2.1 in~\cite{Drnovsek:09}.

\begin{proposition}\label{irr-chars}
  The following are equivalent:
  \begin{enumerate}
  \item\label{irr-self} $\iS$ is ideal irreducible;
  \item\label{irr-ideals} every non-zero algebraic ideal in $\iS$ is
  ideal irreducible;
  \item\label{irr-loc} for any non-zero $x\in X_+$ and $x^*\in X^*_+$
    there exists $S\in\iS$ such that $\langle x^*,Sx\rangle\ne 0$;
  \item\label{irr-ops} $A\iS B\ne\{0\}$ for any non-zero $A,B\in L(X)_+$.
  \item\label{irr-single} for any $x>0$, the ideal generated in $X$ by
    the orbit $\iS x$ is dense in $X$.
  \end{enumerate}
\end{proposition}

\begin{proof}
  The equivalence of \eqref{irr-self} through \eqref{irr-ops} is
  Proposition~2.1 in~\cite{Drnovsek:09}. It is easy to see that
  \eqref{irr-self}$\Rightarrow$\eqref{irr-single}$\Rightarrow$\eqref{irr-loc}.
\end{proof}

\begin{remark}\label{S_r}
  Suppose that $r:=\minrank\iS<+\infty$; let $\iS_r$ be the set of all
  operators of rank $r$ in $\iS$ and zero. Then $\iS_r$ is an ideal,
  so that $\iS$ is ideal irreducible iff $\iS_r$ is ideal
  irreducible. Also, since the set of all operators of rank $r$ is
  closed in $L(X)$, if $\iS$ is $\Rplus$-closed then so is $\iS_r$.
\end{remark}

The following fact was proved in~\cite{Drnovsek:01}, see
also~\cite[Corollary~10.47]{Abramovich:02}. 

\begin{theorem}[\cite{Drnovsek:01}]\label{Tur-pos}
  If $\iS$ consists of compact quasinilpotent operators then $\iS$ is
  ideal reducible.
\end{theorem}

\begin{theorem}\label{irr-exist-proj}
  If $\iS$ is ideal irreducible, $\Rplus$-closed, and contains a
  peripherally Riesz operator then $\minrank\iS<+\infty$ and $\iS$
  contains a projection $P$ with $\rank P=\minrank\iS$.
\end{theorem}

\begin{proof}
  By Proposition~\ref{pR-finrank}, $r:=\minrank\iS$ is finite. By
  Remark~\ref{S_r}, $\iS_r$ is ideal irreducible and, therefore,
  Theorem~\ref{Tur-pos} guarantees that $\iS_r$ contains a
  non-(quasi)\-nilpotent operator. Now apply Lemma~\ref{rank-r-proj}.
\end{proof}

\begin{example}
  The following example shows that, in general, for a peripherally
  Riesz operator $T\in\iS$, the peripheral spectral projection of $T$
  need not be in $\iS$. Let
  \begin{math}
    A=\bigl[
    \begin{smallmatrix}
      1 & 1 \\ 0 & 1
    \end{smallmatrix}
    \bigr]
  \end{math}
  and
  \begin{math}
    B=\bigl[
    \begin{smallmatrix}
      0 & 0 \\ 1 & 0
    \end{smallmatrix}
    \bigr],
  \end{math}
  and let $\iS=\Rplus\{A,B\}$. Clearly, $\iS$ is irreducible and the
  peripheral spectral projection of $A$ is the identity. We claim that
  $I\notin\iS$. Indeed, $\iS$ consists of all positive scalar
  multiples of products of $A$ and $B$ and their limits. Any product
  that involves $B$ has rank one or zero; since the set of matrices of
  rank one or zero is closed, any limit of products involving $B$ is
  also of rank one or zero. On the other hand, it follows from
  \begin{math}
    A^n=\bigl[
    \begin{smallmatrix}
      1 & n \\ 0 & 1
    \end{smallmatrix}
    \bigr]
  \end{math}
  that if $S=\lim b_jA^{n_j}$ then $S$ is a scalar multiple of
  \begin{math}
    \bigl[
    \begin{smallmatrix}
      0 & 1 \\ 0 & 0
    \end{smallmatrix}
    \bigr].
  \end{math}
  Therefore, the only elements of $\iS$ of rank two are the scalar
  multiples of powers of $A$. Hence $I\notin\iS$.
\end{example}

\begin{corollary}\label{comp-min-rank}
  If $\iS$ is ideal irreducible, $\Rplus$-closed, and contains a non-zero
  compact operator then $\minrank\iS<+\infty$ and $\iS$ contains a
  projection $P$ with $\rank P=\minrank\iS$.
\end{corollary}

\begin{proof}
  By Theorem~\ref{irr-exist-proj}, it suffices to show that $\iS$
  contains a non-quasi\-nilpotent compact operator. The set of all
  compact operators in $\iS$ is an ideal, hence is ideal irreducible by
  Proposition~\ref{irr-chars}\eqref{irr-ideals}. Then it contains a
  non-quiasinilpotent operator by Theorem~\ref{Tur-pos}.
\end{proof}

\textbf{Throughout the rest of this section, we assume that $\iS$ is
  an ideal irreducible $\Rplus$-closed semigroup with
  $r:=\minrank\iS<+\infty$.}  We denote by $\iS_r$ for the ideal of
all operators of rank $r$ in $\iS$ or zero; we will write $\iP_r$ for
the (non-empty) set of all projections of rank $r$ in $\iS$.

\begin{lemma}\label{rAS1}
  For every non-zero $S\in\iS_r$ there exists $A\in\iS$ such that $AS$
  is not nilpotent.
\end{lemma}

\begin{proof}
  Let $\iJ=\iS S\iS$.  Then $\iJ$ consists of operators of finite
  rank, hence compact. $\iJ$ is non-zero by
  Proposition~\ref{irr-chars}\eqref{irr-ops} and ideal irreducible by
  Proposition~\ref{irr-chars}\eqref{irr-ideals}.  Hence, by
  Theorem~\ref{Tur-pos}, $\iJ$ contains a non-quasi\-nilpotent
  operator. That is, there exist $A_1,A_2\in\iS$ such that $0\ne
  r(A_1SA_2)=r(A_2A_1S)=r(AS)$ where $A=A_2A_1$.
\end{proof}

Combining this lemma with Lemma~\ref{rank-r-proj}, we show that
$\iS$ contains ``sufficiently many'' rank $r$ projections
(cf. Lemmas~5.2.2 and~8.7.17 in~\cite{Radjavi:00}).

\begin{theorem}\label{r-exists-proj}
  For every $S\in\iS_r$ there exist $P,Q\in\iP_r$ such
  that $PS=SQ=S$.
\end{theorem}

\begin{corollary}\label{faith}
  For every non-zero $x\in X_+$ and $x^*\in X_+^*$ there exist
  $P,Q\in\iP_r$ such that $Qx\ne 0$ and $P^*x^*\ne 0$.
\end{corollary}

\begin{proof}
  Since $\iS_r$ is ideal irreducible, by
  Proposition~\ref{irr-chars}\eqref{irr-loc} there exists $S\in\iS_r$
  such that $x^*(Sx)\ne 0$. Now take $P$ and $Q$ as in
  Theorem~\ref{r-exists-proj}.
\end{proof}

Now, as we know that $\iS$ contains ``sufficiently many'' positive
projections of finite rank, we will need to understand the structure
of the range of such a projection. The following observation is
based on Proposition~11.5 on p.~214 of~\cite{Schaefer:74}.

\begin{numbered}\label{fin-rank-proj}
 \emph{Structure of a positive projection.}  
 Let $P$ be a positive projection on $X$; let $Y=\Range P$. It is
  easy to see that $Y$ is a lattice subspace of $X$ with lattice
  operations $x\overset{*}\wedge y=P(x\wedge y)$ and $x\overset{*}\vee
  y=P(x\vee y)$ for any $x,y\in Y$. We will denote this vector lattice
  by $X_P$. Note that this lattice structure is determined by
  $Y$, so that if $Q$ is another positive projection on $X$ with
  $\Range Q=Y$ then it generates the same lattice structure on $Y$.

  Suppose, in addition, that $n:=\rank P<\infty$. Being a
  finite-dimensional Archimedean vector lattice, $X_P$ is lattice
  isomorphic to $\mathbb R^n$ with the standard order, see, e.g.,
  \cite[Corollary~1, p.~70]{Schaefer:74}. Thus, we can find positive
  *-disjoint $x_1,\dots, x_n\in X_P$ that form a basis of
  $X_P$. Furthermore, we can find positive $y^*_1\dots,y^*_n\in X_P^*$
  such that $y^*_i(x_j)=\delta_{ij}$. Put $x^*_i=y^*_i\circ P$, then
  $x^*_1\dots,x^*_n\in X^*_+$ and $x^*_i(x_j)=\delta_{ij}$. It is easy
  to see that $P=\sum_{i=1}^n x^*_i\otimes x_i$.
\end{numbered}

  Consider $\iS_P=\{PSP_{|X_P}\mid S\in\iS\}$, so that $\iS_P\subseteq
  L_+(X_P)$ (note that $P$ need not be in $\iS$). The following
  proposition extends Lemmas~5.2.1 and~8.7.16 in~\cite{Radjavi:00}.

\begin{proposition}\label{S_P}
  If $P$ is a positive finite-rank projection and $P\iS P\subseteq\iS$
  then $\iS_P$ is an irreducible $\Rplus$-closed semigroup in $L_+(X_P)$.
\end{proposition}

\begin{proof}
  It follows from $P\iS P\subseteq\iS$ that $\iS_P$ is a
  semigroup. Let $P=\sum_{i=1}^n x^*_i\otimes x_i$ as before; relative
  to the basis $x_1$,\dots,$x_n$, we can view $\iS_P$ as a semigroup
  of positive $n\times n$ matrices.  Since $\iS$ is ideal irreducible, by
  Proposition~\ref{irr-chars}\eqref{irr-loc}, for each $i,j$ there
  exists $S\in\iS$ such that $x^*_i(Sx_j)\ne 0$, i.e., the $(ij)$-th
  entry of the matrix of $PSP_{|X_P}$ is non-zero. Hence, $\iS_P$ is
  irreducible by Proposition~\ref{irr-chars}\eqref{irr-loc}.

  To show that $\iS_P$ is closed, suppose that $PS_nP_{|X_P}\to A$ for
  some sequence $(S_n)$ in $\iS$ and some $A\in L(X_P)$. Put $S=PAP\in
  L(X)$.  Then $PS_nP\to S$, so that $S\in\iS$ because $\iS$ is
  closed.  Now $A=PSP_{|X_P}$ yields $A\in\iS_P$.
\end{proof}

Of course, the assumption that $P\iS P\subseteq\iS$ is satisfied when
$P\in\iS$. If, in addition, $\rank P=r$, we get the following much
stronger result. We write $\iG_P:=\bigl\{PSP_{|X_P}\mid S\in\iS\mbox{
    and } r(PSP)=1\bigr\}$.

\begin{proposition}\label{S_P-perm}
  Suppose that $P\in\iP_r$. Then every non-zero element of $\iS_P$ is
  invertible and, after appropriately scaling the basis vectors of
  $X_P$, $\iG_P$ is a transitive%
  \footnote{Transitive in the sense that for each $i$ and $j$ there
    exists $A\in\iG_P$ such that $Ax_i=x_j$.}  group of permutation
  matrices.
\end{proposition}

\begin{proof}
  By Proposition~\ref{S_P}, $\iS_P$ is irreducible and
  $\Rplus$-closed.  Since $r=\minrank\iS$, every non-zero element of
  $\iS_P$ is invertible. It follows from Lemma~\ref{matrix-group} that
  $\iG_P$ is a group. In particular, each matrix in $\iG_P$ has a
  positive inverse. It is known that a positive matrix $A$
  in $M_r(\mathbb R)$ has a positive inverse iff it is a weighted
  permutation matrix with positive weights, i.e., there exist positive
  weights $w_1,\dots,w_r$ and a permutation $\sigma$ of
  $\{1,\dots,r\}$ such that $Ax_i=w_ix_{\sigma(i)}$ for each
  $i=1,\dots,r$.

  It is left to show that, after scaling $x_i$'s, we may assume that
  all the weights are equal to one (for all $S\in\iG_P$). We
  essentially follow the proof of Lemma~5.1.11 in~\cite{Radjavi:00}.
  Since $\iS_P$ is an irreducible semigroup of matrices, for each
  $i,j\le r$ there exists $A\in\iS_P$ such that $Ax_i$ is a scalar
  multiple of $x_j$. Put $A_1=I$. For each $2=1,\dots,r$ fix
  $A_i\in\iG_P$ such that $A_ix_1=\mu_ix_i$ for some
  $\mu_i>0$. Replacing $x_i$ with $\mu_ix_i$ for $i=2,\dots,r$, we
  have $A_ix_1=x_i$. It suffices to show that with respect to these
  modified $x_i$'s, all the matrices in $\iG_P$ are permutation
  matrices. Let $B\in\iG_P$. We know that $B$ is a weighted
  permutation matrix. Take any $i$ and $j$ such that $\lambda:=b_{ij}$
  is non-zero. Put $C=A_i^{-1}BA_j$. Then $C\in\iG_P$ and
  $Cx_1=\lambda x_1$, so that $\lambda=c_{11}\le r(C)=1$. Similarly,
  $\lambda^{-1}$ is the $(1,1)$'s entry of $C^{-1}$, hence
  $\lambda^{-1}\le 1$ as well, so that $\lambda=1$.

  Finally, transitivity of $\iG_P$ follows from the irreducibility of $\iS_P$.
\end{proof}

\begin{remark}\label{x0}
  It follows that the vector $x_0=x_1+\dots+x_r$ is invariant under
  $\iG_P$. Furthermore, for each $S\in\iS$, if $PSP\ne 0$ then the
  minimality of rank implies that $PSP$ is an isomorphism on $X_P$, so
  that $r(PSP)\ne 0$ and, therefore, a scalar multiple of $PSP$ is in
  $\iG_P$. It follows that $x_0$ is a common eigenvector for $\iS_P$
  with $PSPx_0=r(PSP)x_0$.
\end{remark}

\section{Semigroups with all the rank $r$ projections having the same
  range}
\label{sec:sr}

As in the previous section, $\iS$ will stand for an $\Rplus$-closed
ideal irreducible semigroup of positive operators on a Banach lattice, with
$r:=\minrank\iS<\infty$. We will write $\iS_r$ for the (ideal irreducible)
ideal of all operators of rank $r$ in $\iS$ and zero, and $\iP_r$ for
the set of all projections of rank $r$ in $\iS$ (which is non-empty by,
e.g., Corollary~\ref{comp-min-rank}).

Let $P\in\iP_r$ and $x_0$ be as in Remark~\ref{x0}.  For $x_0$ to be a
common eigenvector of the entire semigroup $\iS$ it would suffice that
$\Range P$ is invariant under $S$ and that $PSP\ne 0$ for every
non-zero $S\in\iS$. We will see that, surprisingly, the former implies
the latter. The following proposition extends Lemmas~5.2.4 and~8.7.18
in~\cite{Radjavi:00}.

\begin{proposition}\label{sr}
  The following are equivalent.
  \begin{enumerate}
  \item\label{sr-proj} All projections in $\iP_r$ have the same range;
  \item\label{sr-Sr} All non-zero operators in $\iS_r$ have the same range;
  \item\label{sr-pres} $S(\Range P)=\Range P$ for all non-zero $S\in\iS$ and
    $P\in\iP_r$;
  \item\label{sr-some-inv} The range of some $P\in\iP_r$ is $\iS$-invariant;
  \end{enumerate}
\end{proposition}

\begin{proof}
  \eqref{sr-proj}$\Rightarrow$\eqref{sr-Sr} follows from
  Theorem~\ref{r-exists-proj}.

  \eqref{sr-Sr}$\Rightarrow$\eqref{sr-pres} Let $S\in\iS$ and
  $P\in\iP_r$.  Since $\iS_r$ is ideal irreducible, $S\iS_r\ne\{0\}$, so
  that $ST\ne 0$ for some $T\in\iS_r$. It follows from $\Range
  T=\Range P$ that $SP\ne 0$. Since $SP\in\iS_r$, we have $\Range
  SP=\Range P$.

  \eqref{sr-pres}$\Rightarrow$\eqref{sr-some-inv} is trivial.
  
  \eqref{sr-some-inv}$\Rightarrow$\eqref{sr-proj} Suppose that $\Range
  P$ is $\iS$-invariant for some $P\in\iP_r$. Take any $Q\in\iP_r$. We
  have $Q\iS P\ne\{0\}$ by Proposition~\ref{irr-chars}\eqref{irr-ops},
  so that $QSP\ne 0$ for some $S\in\iS$. By assumption, $SP=PSP$, so
  that $QPSP\ne 0$, hence $QP\ne 0$. This yields $\rank QP=r$. By
  assumption, $\Range QP=Q(\Range P)\subseteq\Range P$, but,
  trivially, $\Range QP\subseteq\Range Q$. Since all the three ranges
  are $r$-dimensional, the inclusions are, in fact, equalities, so
  that $\Range P=\Range QP=\Range Q$.
\end{proof}

Next, we would like to provide a few examples.

\begin{example}\label{ex-2-min-proj}
  Suppose that $x,y\in X_+$ and $x^*,y^*\in X^*_+$ such that
  $x^*(x)=y^*(x)=x^*(y)=y^*(y)=1$. Let $\iS_1=\{x^*\otimes x,
  y^*\otimes x, x^*\otimes y, y^*\otimes y\}$. Then $\iS_1$ is a
  semigroup of projections. Let $\iS=\Rplus\iS_1$, the semigroup of
  all positive scalar multiples of the elements of $\iS_1$. Clearly,
  $\iS_1$ is exactly the set of the minimal rank projections in $\iS$,
  and the ranges of the elements of $\iS$ are $\Span x$ and $\Span
  y$. In particular, all the ranges are the same iff $x=y$.
\end{example}

\begin{example}\label{ex:no-sr}
  More specifically, take in Example~\ref{ex-2-min-proj} $X=\mathbb
  R^2$,
  \begin{math}
    x=\Bigl[
    \begin{smallmatrix}
      \frac{1}{2} \\ \frac{1}{2} 
    \end{smallmatrix}
    \Bigr],
  \end{math}
  \begin{math}
    y=\Bigl[
    \begin{smallmatrix}
      \frac{1}{3} \\ \frac{2}{3} 
    \end{smallmatrix}
    \Bigr],
  \end{math}
  and $x^*=y^*=[1,1]$. Then $\iP_r=\iS_1=\{P,Q\}$ where 
 \begin{math}
    P=\Bigl[
    \begin{smallmatrix}
      \frac{1}{2} &  \frac{1}{2} \\
      \frac{1}{2} &  \frac{1}{2}
    \end{smallmatrix}
    \Bigr]
  \end{math}
  and
  \begin{math}
    Q=\Bigl[
    \begin{smallmatrix}
      \frac{1}{3} &  \frac{1}{3} \\
      \frac{2}{3} &  \frac{2}{3}
    \end{smallmatrix}
    \Bigr]
  \end{math}
  are ideal irreducible and 
  have different ranges.
\end{example}

\begin{example}\label{ex:non-uniq-proj}
  Again in Example~\ref{ex-2-min-proj}, take  $X=\mathbb
  R^2$,
  \begin{math}
    x=y=\bigl[
    \begin{smallmatrix}
      1 \\ 1
    \end{smallmatrix}
    \bigr],
  \end{math}
  $x^*=\bigl[\frac{1}{2},\frac{1}{2}\bigr]$, and
  $y^*=\bigl[\frac{1}{3},\frac{2}{3}\bigr]$.
  Then $\iP_r=\iS_1=\{P,Q\}$ where 
 \begin{math}
    P=\Bigl[
    \begin{smallmatrix}
      \frac{1}{2} &  \frac{1}{2} \\
      \frac{1}{2} &  \frac{1}{2}
    \end{smallmatrix}
    \Bigr]
  \end{math}
  and
  \begin{math}
    Q=\Bigl[
    \begin{smallmatrix}
      \frac{1}{3} &  \frac{2}{3} \\
      \frac{1}{3} &  \frac{2}{3}
    \end{smallmatrix}
    \Bigr]
  \end{math}
  are both irreducible and have the same range.
\end{example}

\begin{example}\label{ex:non-irr-proj}
  Again in Example~\ref{ex-2-min-proj}, take  $X=\mathbb
  R^2$,
  \begin{math}
    x=y=\bigl[
    \begin{smallmatrix}
      1 \\ 1
    \end{smallmatrix}
    \bigr],
  \end{math}
  $x^*=[1,0]$, and
  $y^*=[0,1]$. 
  Then $\iP_r=\iS_1=\{P,Q\}$ where 
 \begin{math}
    P=\bigl[
    \begin{smallmatrix}
      1 & 0 \\
      1 & 0
    \end{smallmatrix}
    \bigr]
  \end{math}
  and
  \begin{math}
    Q=\bigl[
    \begin{smallmatrix}
      0 & 1 \\
      0 & 1
    \end{smallmatrix}
    \bigr].
  \end{math}
   Even though neither $P$ nor $Q$ are irreducible, they generate an
  irreducible semigroup. Note that $P$ and $Q$ have the same range.
\end{example}

\textbf{For the rest of this section, we assume that all the projections
  in $\iP_r$ have the same range}. This condition looks
rather strong at the first glance. However, it will follow immediately
from Proposition~\ref{center} that it is satisfied for commutative
semigroups, and, in particular, for semigroups generated by a single
operator.

We are now going to prove a Banach lattice version of Lemmas~5.2.5
and~8.7.9 as well as Theorems~5.2.6 and~8.7.20
of~\cite{Radjavi:00}. Denote by $Y$ the common range of the
projections in $\iP_r$.  For a non-zero $S\in\iS$ we denote by $S_Y$
the restriction of $S$ to $Y$; we write $\iS_Y=\{S_Y\mid 0\ne
S\in\iS\}$ and $\iG:=\bigl\{S_Y\mid S\in\iS,\ r(S_Y)=1\bigr\}$. Note
that $\iS_Y=\iS_P$ and $\iG=\iG_P$ for every $P\in\iP_r$,
cf. \ref{fin-rank-proj} and Proposition~\ref{S_P-perm}. In particular,
$\iG$ is a transitive group of permutation matrices in the appropriate
positive basis $x_1,\dots,x_r$ of $Y$. The following lemma follows
immediately from Proposition~\ref{sr}\eqref{sr-pres}.

\begin{lemma}\label{main-group}
  For each non-zero $S\in\iS$, the restriction $S_Y$ is an isomorphism
  of $Y$. In particular, $r(S_Y)>0$ and $\frac{1}{r(S_Y)}S_Y\in\iG$.
\end{lemma}

It follows, in particular, that $\iS$ contains no zero divisors and no
non-zero quasi-nilpotent operators.

\begin{theorem}\label{main}
  There exist disjoint positive vectors $x_1,\dots,x_r$ such that
  every $S\in\iS$ acts as a scalar multiple of a permutation on
  $x_i$'s.
\end{theorem}

\begin{proof}
  The statement follows immediately from Lemma~\ref{main-group} and
  Proposition~\ref{S_P-perm} except for the disjointness of $x_i$'s.
  By \ref{fin-rank-proj}, we know that $Y$ is a lattice subspace of
  $X$, and the positive vectors $x_1,\dots,x_r$ form a basis of $Y$
  and are disjoint in $Y$.  The latter means that for each $i,j\le r$
  we have $P(x_i\wedge x_j)=0$ for every $P\in\iP_r$. It now follows
  from Corollary~\ref{faith} that $x_i\perp x_j$ in $X$.
\end{proof}

Note that the ideal generated by $Y$ is invariant under $\iS$.

\begin{corollary}\label{ideal-dense}
  The subspace $Y$ is a
  non-zero finite-dimensional sublattice of $X$ invariant under
  $\iS$\footnote{This can be viewed as a Banach lattice version of
    results in \cite{Radjavi:08}.}.
  The ideal generated by $Y$ is dense in $X$ 
\end{corollary}

\begin{corollary}\label{sr-eigen}
  All the operators in $\iS$ have a unique common eigenvector
  $x_0$. Namely, $Sx_0=r(S_Y)x_0$ for each $S\in\iS$. Furthermore,
  $x_0$ is positive and quasi-interior.
\end{corollary}

\begin{proof}
  Let $x_1,\dots,x_r$ be as in the theorem. Put
  $x_0=x_1+\dots+x_r$. Since each $S\in\iS$ is just a scalar multiple
  of a permutation on $x_i$'s, it follows that $x_0$ is a common
  eigenvector for $\iS$.  The ideal $I_{x_0}$ generated by $x_0$ is
  exactly the ideal generated by $Y$, hence is dense in $X$; it
  follows that $x_0$ is quasi-interior. It is left to verify
  uniqueness (of course, up to scaling). Indeed, suppose that $y$ is
  also a common eigenvector for $\iS$. Then for each $P\in\iP_r$ we
  have $y\in\Range P=Y$. It follows that $y$ is a linear combination
  of $x_i$'s. In particular, viewed as an element of $\mathbb R^r$, it
  is a common eigenvector of the transitive group of permutations
  $\iG$, so that it has to be of the form $(\lambda,\dots,\lambda)$;
  it follows that $y=\lambda x_0$.
 \end{proof}

Note that the semigroup in Example~\ref{ex:no-sr} has no common eigenvectors.

\begin{numbered}
  \emph{Other eigenvalues of $\iS$.}
  Since every element of $\iG$ is a permutation matrix with respect to
  the basis $x_1,\dots,x_r$ of $Y$, its Jordan form is diagonal and
  unimodular. It follows that every non-zero $S\in\iS$ has at least
  $r$ eigenvalues of modulus $r(S_Y)$ (counting geometric
  multiplicities). If we scale $S$ so that $r(S_Y)=1$ then
  $(S_Y)^{r!}$ is the identity of $Y$; it follows that these
  eigenvalues satisfy $\lambda^{r!}=1$.
\end{numbered}

\begin{numbered}\label{block-matrix}
  \emph{Block-matrix structure of $\iS$.} Let $X_i=\overline{I_{x_i}}$
  for each $i=1,\dots,r$. Then $X=X_1\oplus\dots\oplus X_r$ is a
  decomposition of $X$ into pair-wise disjoint closed ideals, and for
  every non-zero $S\in\iS$ the block-matrix of $S$ with respect to
  this decomposition has exactly one non-zero block in each row and in
  each column.
\end{numbered}

\begin{proposition}\label{sr-pR}
  If $T\in\iS$ is peripherally Riesz then $r(T_Y)=r(T)$. Furthermore,
  if $r(T)=1$ then the component of $T$ corresponding to $\sigma_{\rm
    per}(T)$ is unimodular.
\end{proposition}

\begin{proof}
  Without loss of generality, $r(T)=1$. By Lemma~\ref{main-group},
  $\iS$ has no non-zero nilpotent elements. It follows that the
  nilpotent case in Proposition~\ref{dych-proj-nilp} is impossible,
  hence the peripheral spectral projection $P$ of $T$ is in $\iS$ and
  there is an increasing sequence $(m_j)$ in $\mathbb N$ with
  $T^{m_j}\to P$. In particular, $(T_Y)^{m_j}\to P_Y$.  It follows
  from $r(T)=1$ that $r(T_Y)\le 1$. Suppose that $r(T_Y)<1$. Then
  $(T_Y)^{m_j}\to 0$, hence $P_Y=0$.  But this contradicts $P_Y$ being
  an isomorphism by Lemma~\ref{main-group}.
\end{proof}

\begin{corollary}
  If every non-zero operator in $\iS$ is peripherally Riesz then
  spectral radius is multiplicative on $\iS$.
\end{corollary}

\begin{proof}
  Let $S,T\in\iS$. By Proposition~\ref{sr-pR}, $r(S)=r(S_Y)$,
  $r(T)=r(T_Y)$, and $r(ST)=r(S_YT_Y)$. Since $S_Y$ and $T_Y$ are
  scalar multiples of permutation matrices by Theorem~\ref{main}, it
  follows that $r(S_YT_Y)=r(S_Y)r(T_Y)$.
\end{proof}

For each non-zero $S\in\iS$ we have $r(S^*)=r(S)\ge r(S_Y)>0$ by
Lemma~\ref{main-group}. The following is a refinement of this fact.

\begin{corollary}\label{sr-loc-r}
  For every non-zero $S\in\iS$ and $x^*\in X^*_+$, we have 
  \begin{math}
    \liminf_n\norm{S^{*n}x^*}^{\frac{1}{n}}\ge r(S_Y).
  \end{math}
  In particular, $S^*$ is strictly positive.
\end{corollary}

\begin{proof}
  For each $n$, we have
  \begin{math}
    (S^{*n}x^*)(x_0)=x^*(S^nx_0)=r(S_Y)^nx^*(x_0)
  \end{math}
  by Corollary~\ref{sr-eigen}. Since $x_0$ is quasi-interior, we have
  $x^*(x_0)\ne 0$, so that
  \begin{math}
    r(S_Y)^n\le\frac{\norm{x_0}}{x^*(x_0)}\norm{S^{*n}x^*}.
  \end{math}
  The result is now straightforward.
\end{proof}

\begin{remark}\label{dual-disj}
  Let $x_1,\dots,x_r$ be a disjoint positive basis of $Y$ as before. Suppose
  that $P\in\iP_r$, then, as in~\ref{fin-rank-proj}, we have
  $P=\sum_{i=1}^rx_i^*\otimes x_i$ for some positive functionals
  $x_1^*\dots,x_r^*$. Observe that these functionals are
  disjoint. Indeed, by Riesz-Kantorovich formula, if $i\ne j$ then
  \begin{displaymath}
    (x_i^*\wedge x_j^*)(x_0)
    =\inf\bigl\{x_i^*(u)+x_j^*(v)\mid u,v\in[0,x_0], u+v=x_0\bigr\}
    \le x_i^*(x_j)+x_j^*(x_0-x_j)
    =0,
  \end{displaymath}
  hence $(x_i^*\wedge x_j^*)(x_0)=0$. Since $x_0$ is quasi-interior,
  it follows that $x_i^*\wedge x_j^*=0$.
\end{remark}

\section{Semigroups with a unique rank $r$ projection}
\label{sec:up}

As before, we assume that $\iS$ is an ideal irreducible
$\Rplus$-closed semigroup of positive operators on a Banach lattice
$X$ with $r=\minrank\iS<+\infty$.

In the previous section we showed that if all the rank $r$
projections have the same range then $\iS$ has some nice
properties. In this section, we will show that many of these
properties are also enjoyed by the dual semigroup $\iS^*=\{S^*\mid
S\in\iS\}$ provided that $\iS$ has a \emph{unique} projection of rank
$r$. Even though this is, obviously, a stronger assumption, the
following proposition implies that it is still satisfied for
commutative semigroups. It is analogous to Lemmas~5.2.7 and~8.7.21
of~\cite{Radjavi:00}.

\begin{proposition}\label{center}
  The following are equivalent:
  \begin{enumerate}
  \item\label{center-unique} $\iP_r$ consists of a single projection;
  \item\label{center-all} Every $P\in\iP_r$ commutes with $\iS$;
  \item\label{center-some} Some $P\in\iP_r$ commutes with $\iS$.
  \end{enumerate}
\end{proposition}

\begin{proof}
  \eqref{center-unique}$\Rightarrow$\eqref{center-all} Suppose that
  $\iP_r=\{P\}$ and let $0\ne S\in\iS$. It follows from
  Proposition~\ref{sr}\eqref{sr-pres} that $PSP\ne 0$. Hence, $PS$ and
  $SP$ are non-zero elements of $\iS_r$. Applying
  Theorem~\ref{r-exists-proj} to $PS$ and $SP$ we get
  $PS=PSP=SP$.

  \eqref{center-all}$\Rightarrow$\eqref{center-some} is trivial.

  \eqref{center-some}$\Rightarrow$\eqref{center-unique} Suppose
  $P\in\iP_r$ commutes with $\iS$. It follows that $PSP=SP$ for all
  $S\in\iS$, hence by Proposition~\ref{sr}\eqref{sr-some-inv}, all the
  projections in $\iP_r$ have the same range. Therefore,
  $P=QP=PQ=Q$ for every $Q\in\iP_r$.
\end{proof}

Recall that by Proposition~\ref{S_P-perm} and Remark~\ref{x0}, for
each $P\in\iP_r$, there is a basis $x_1,\dots,x_n$ of $X_P=\Range P$
such that the group $\iG_P$ can be viewed as a transitive group of
permutations of the vectors $x_1,\dots,x_r$; it follows that
$x_0=x_1+\dots+x_n$ is a common eigenvector of every operator in $\iS$
which leaves $\Range P$ invariant. Then we observed in
Section~\ref{sec:sr} that if all the projections in $\iP_r$ have the
same range, then this range is invariant under all operators in $\iS$
and, therefore, $x_0$ is a common eigenvector for $\iS$. 

\textbf{Throughout the rest of the section, we assume that $\iS$ has a
  unique projection $P$ of rank $r$.} This condition allows us to
``dualize'' the results of Section~\ref{sec:sr} for $\iS^*$, even
though $\iS^*$ may not be ideal irreducible.

Suppose $\iP_r=\{P\}$. As in Section~\ref{sec:sr}, we denote $Y=\Range
P=X_P$. We can write it as $P=\sum_{i=1}^rx_i^*\otimes x_i$ as in
Remark~\ref{dual-disj}. It is easy to see that $P^*$ is a projection
onto $X_{P^*}:=\Range P^*=\Span\{x_1^*,\dots,x_r^*\}$ in $X^*$.  For
every non-zero $S\in\iS$, it follows from Proposition~\ref{center}
that $PSP=SP=PS$, so that $P^*S^*P^*=S^*P^*$, and, therefore,
$X_{P^*}$ in invariant under $S^*$. Note that
$r(P^*S^*P^*)=r(PSP)=r(S_Y)\ne 0$ by Lemma~\ref{main-group}. As in
Section~\ref{sec:sr}, if $r(S_Y)=1$ then $S\in\iG$ (since $P$ is
unique, we write $\iG_P=\iG$) and $S$ acts as a permutation matrix on
$x_1,\dots,x_r$. It follows from $x_i^*(x_j)=\delta_{ij}$ that $S^*$
acts as a permutation matrix on $x_1^*,\dots,x_r^*$ (namely, as the
transpose of the matrix of $S$ on $x_1,\dots,x_r$). Moreover, since
$\iG$ is transitive on $x_1,\dots,x_r$, the group $\iG^*:=\{S^*\mid
S\in\iG\}$ is transitive on $x_1^*,\dots,x_r^*$. In particular, we
have $S^*x_0^*=x_0^*$, where $x_0^*=x_1^*+\dots+x_r^*$.

\begin{corollary}\label{dual-eigen}
  For every non-zero $S\in\iS$, the operator $\frac{1}{r(S_Y)}S^*$ acts
  as a permutation of $x_1^*,\dots,x_r^*$. In particular,
  $S^*x_0^*=r(S_Y)x_0^*$ for each non-zero $S\in\iS$. The functional
  $x_0^*$ is strictly positive and is a unique common eigenfunctional for $\iS^*$.
\end{corollary}

\begin{proof}
  Uniqueness is proved exactly as in Corollary~\ref{sr-eigen}.  It is
  left to prove that $x_0^*$ is strictly positive. Fix $x>0$. By
  Proposition~\ref{irr-chars}\eqref{irr-loc}, there exists $S\in\iS$
  with $x_0^*(Sx)\ne 0$. Since $r(S_Y)\ne 0$ by Lemma~\ref{main-group}
  and $x_0^*(Sx)=(S^*x_0^*)x=r(S_Y)x_0^*(x)$, we have $x_0^*(x)\ne 0$.
\end{proof}

In view of Corollary~\ref{dual-eigen}, the following fact is the
dual version of  Corollary~\ref{sr-loc-r}; the proof is
analogous. Corollaries~\ref{dual-eigen} and~\ref{up-loc-r} extend
Lemma~5.2.8 and Corollary~8.7.22 in~\cite{Radjavi:00}.

\begin{corollary}\label{up-loc-r}
  For every $x>0$ and every non-zero $S\in\iS$ we have
  $\liminf_n\norm{S^nx}^{\frac{1}{n}}\ge r(S_Y)$. In particular, $S$
  is strictly positive.
\end{corollary}

This means that not only every non-zero $S\in\iS$ is not
quasi-nilpotent, but it is not even \emph{locally} quasi-nilpotent.

We would like to point out that Corollaries~\ref{dual-eigen}
and~\ref{up-loc-r} generally fail if instead of assuming that $\iS$
has a unique minimal projection we only assume, as in
Section~\ref{sec:sr}, that all the rank $r$ projections in $\iS$ have
the same range.  Indeed, the semigroups in
Examples~\ref{ex:non-uniq-proj} and~\ref{ex:non-irr-proj} are
irreducible, $\Rplus$-closed, have exactly two distinct projections
$P$ and $Q$ of rank $r$ each, and they have the same
range. Nevertheless it is easy to see that the dual semigroup $\iS^*$
in Example~\ref{ex:non-uniq-proj} has no common eigenfunctionals (as
$P^*$ and $Q^*$ have no common eigenfunctionals), while the operators
$P$ and $Q$ in Example~\ref{ex:non-irr-proj} are not strictly
positive.

Recall that a positive operator $T$ is \term{strongly expanding} if
$Tx$ is quasi-interior whenever $x>0$.

\begin{corollary}
  The projection $P$ is strongly expanding iff $r=1$.
\end{corollary}

\begin{proof}
  Note that $P$ is strictly positive by Corollary~\ref{up-loc-r},
  and the ideal generated by $\Range P$ is dense in $X$ by
  Corollary~\ref{ideal-dense}.  If $r=1$ then $\Range P$ is the span
  of $x_1$, hence $x_1$ is quasi-interior and $Px$ is a positive
  scalar multiple of $x_1$ whenever $x>0$.  On the other hand, if
  $r>1$ then $Px_1=x_1\perp x_2$, hence $Px_1$ is not quasi-interior.
\end{proof}

The following proposition should be compared with
Proposition~\ref{sr-pR}.

\begin{proposition}\label{r-sublat}
  Let $0\ne S\in\iS$. If $r(S)$ is an eigenvalue of $S$ or $S^*$ then
  $r(S_Y)=r(S)$, and the eigenspace is a sublattice.
\end{proposition}

\begin{proof}
  Suppose that $Sx=r(S)x$ for some $x\ne 0$. It follows from
  $r(S_Y)\le r(S)$ that
  \begin{equation}
    \label{eq:up-compar}
    r(S_Y)\abs{x}\le r(S)\abs{x}=\abs{Sx}\le S\abs{x},
  \end{equation}
  so that $S\abs{x}-r(S_Y)\abs{x}\ge 0$. On the other hand,
  Corollary~\ref{dual-eigen} yields
  $x_0^*\bigl(S\abs{x}-r(S_Y)\abs{x}\bigr)=0$. Since $x_0^*$
  is strictly positive, we have $S\abs{x}=r(S_Y)\abs{x}$. Combining
  this with~\eqref{eq:up-compar}, we get $r(S_Y)=r(S)$. It
  also follows that $\abs{x}$ is also in the eigenspace, so that the
  eigenspace is a sublattice.

  The proof in the case when $r(S)$ is an eigenvalue of $S^*$ is
  similar in view of the fact that $x_0$ is quasi-interior and,
  therefore, acts as a strictly positive functional on $X^*$.
\end{proof}

\begin{example}
  Fix $n>2$ and let $\iS$ be the semigroup of all positive
  scalar multiples of all permutation matrices in $M_n(\mathbb
  R)$. Then $\iS$ is not commutative; nevertheless, the identity
  matrix is the unique element of $\iP_r$.
\end{example}

\subsection*{Commutative semigroups}
All the results of Sections~\ref{sec:sr} and~\ref{sec:up} apply to
commutative semigroups. In particular, the group $\iG$ is a
commutative transitive semigroup of permutation matrices. Every matrix
in such a group is a direct sum of cycles of equal lengths; it
follows, in particular, that $S^r_{|Y}$ is a multiple of the identity
on $Y$ for each $S\in\iS$. See \cite[Lemma~5.2.11]{Radjavi:00}) for a
proof and further properties of such groups of matrices.

\section{Applications to finitely generated semigroups}
\label{sec:app}

\subsection*{Singly generated semigroups}
Suppose that $T$ is a positive ideal irreducible peripherally Riesz
operator on a Banach lattice $X$. We now present a version of
Perron-Frobenius Theorem for $T$, extending Corollaries~5.2.3
and~8.7.24 in~\cite{Radjavi:00}. In addition, we completely describe
$\Rplus T$ (cf. Proposition~\ref{dych-proj-nilp}). For simplicity,
scaling $T$ if necessary, we assume that $r(T)=1$. Let $X=X_1\oplus
X_2$ be the spectral decomposition for $T$ where $X_1$ is the subspace
for $\sigma_{\rm per}(T)$, and $T=T_1\oplus T_2$ the corresponding
decomposition of $T$. Clearly, $\Rplus T$ is ideal irreducible.  Since
it is commutative, all the results of Sections~\ref{sec:sr}
and~\ref{sec:up} apply to it. We will see that, surprisingly, the
asymptotic part of $\Rplus T$ is very small: it consists of finitely
many operators and their positive scalar multiples.

\begin{theorem}\label{single}
  Under the preceding assumptions, $\dim X_1=\minrank\Rplus T$, $X_1$
  has a basis of disjoint positive vectors $x_1,\dots,x_r$ such that
  $T_1$ is a cyclic permutation of $x_1,\dots,x_r$, and $\Rplus T$
  consists precisely of all the powers of $T$, of the operators
  $T^k_1\oplus 0$ for $k=0,\dots,r-1$, and of their positive scalar
  multiples (and zero).
\end{theorem}

\begin{proof}
  By Proposition~\ref{sr-pR}, $T_1$ is unimodular. Hence, we are in
  the unimodular case of Proposition~\ref{dych-proj-nilp}. In
  particular, the peripheral spectral projection $P$ is the only
  projection in the semigroup. It follows that $r:=\minrank\Rplus
  T=\dim X_1$, $\iP_r=\{P\}$, and $X_1$ coincides with $Y$ in the
  notation of Section~\ref{sec:sr}. This implies by Theorem~\ref{main}
  and Corollary~\ref{ideal-dense} that $X_1$ is a sublattice generated
  by some disjoint sequence $x_1,\dots,x_r$ and 
  $T_1$ is a permutation of $x_i$'s. We claim that this
  permutation is a cycle of full length $r$. Indeed, otherwise, $T_1$
  has a cycle of length $m<r$, i.e., after re-numbering the basis
  vectors, $T_1$ acts as a cycle on $x_1,\dots,x_m$. But then the
  closed ideal generated by $x_1,\dots,x_m$ is invariant under $T$
  and is proper as it is disjoint with $x_{m+1},\dots,x_r$.

  It follows that $T_1^r$ is the identity of $X_1$, so that
  the set of the distinct powers of $T_1$ is, in fact, finite. Suppose that
   $0\ne S=\lim_jb_jT^{n_j}$ for some $(b_j)$ in $\mathbb R_+$
   and some strictly increasing $(n_j)$ in $\mathbb N$. By
   Proposition~\ref{dych-proj-nilp}, $S_|{X_2}=0$ and
   $S_{|X_1}=\lim_jb_jT_1^{n_j}$. Since the set of the distinct powers of
   $T_1$ is finite, it follows that $S_{|X_1}$ is a scalar multiple of a
   power of $T_1$.
\end{proof}

\begin{remark}
  \begin{enumerate}
  \item The ideal generated by $X_1$ is, clearly, invariant under
  $T$, hence it is dense in $X$. 
  \item $X_1$ is a non-zero finite-dimensional sublattice
    invariant under $T$.
  \item We observed in the proof that $P$ is
  the unique projection in the semigroup; it can, actually, be viewed as
  $T_1^0\oplus 0$.
  \end{enumerate}
\end{remark}

\begin{remark}\label{dP-comp-pR}
  Suppose that $S$ is a positive ideal irreducible operator
  such that $S^m$ is compact for some $m$. Then $S$ is strictly
  positive by \cite[Theorem~9.3]{Abramovich:02}, hence $S^m\ne 0$.
  Applying \cite[Theorem~9.19]{Abramovich:02} to $S^m$ we conclude
  that $r(S^m)\ne 0$ and, therefore, $r(S)\ne 0$. It follows that $S$ is
  peripherally Riesz. Therefore, Theorem~\ref{single} applies to
  positive ideal irreducible power compact operators.
\end{remark}

\begin{remark}
  It has been known (see, e.g.,~\cite{Niiro:66}) that if $T$ is a
  positive ideal irreducible peripherally Riesz operator then 
  $r(T)>0$, $\sigma_{\rm per}(T)=r(T)G$ where $G$ is the set of all
  $m$-th roots of unity for some $m\in\mathbb N$, and each point in
  $\sigma_{\rm per}(T)$ is a simple pole of the resolvent with
  one-dimensional eigenspace. This can now be easily deduced from
  Theorem~\ref{single}.
\end{remark}

\subsection*{Semigroups generated by two commuting operators}

de Pagter showed in~\cite{dePagter:86} that every ideal irreducible
positive compact operator on a Banach lattice has strictly positive
spectral radius. This was extended
in~\cite[Corollary~4.11]{Abramovich:92a} (see also Corollary~9.21
in~\cite{Abramovich:02}) to a pair of operators as follows: suppose
that $S$ and $K$ are two non-zero positive commuting operators such
that $S$ is ideal irreducible and $K$ is compact, then $r(S)>0$ and
$r(K)>0$. Moreover, $K$ is not even locally quasinilpotent at any
positive non-zero vector $x$, i.e.,
\begin{math}
  \liminf_n\norm{K^nx}^{\frac{1}{n}}>0,
\end{math}
see e.g., \cite[Corollary~9.19]{Abramovich:02}.
Using  the results of the preceding sections, we can
now strengthen this conclusion even further.

\begin{theorem}\label{two-op-ii}
  Under the preceding assumptions on $S$ and $K$,
  there exists a quasi-interior vector $x_0\in X_+$, a
  strictly positive functional $x_0^*$, and a positive real $\lambda$ such
  that $Sx_0=\lambda x_0$, $S^*x_0^*=\lambda x_0^*$, $Kx_0=r(K)x_0$, and
  $K^*x_0^*=r(K)x_0^*$. Furthermore, 
  \begin{math}
    \liminf_n\norm{S^nx}^{\frac{1}{n}}\ge\lambda
  \end{math}
  and
  \begin{math}
    \lim_n\norm{K^nx}^{\frac{1}{n}}=r(K)>0
  \end{math}
  whenever $x>0$.
\end{theorem}

\begin{proof}
  Let $\iS=\Rplus\{S,K\}$. Then $\iS$ is ideal irreducible and
  commutative, so that all the results of Sections~\ref{sec:sr}
  and~\ref{sec:up} apply. In particular, by Corollaries~\ref{sr-eigen}
  and~\ref{dual-eigen} there exist a quasi-interior vector $x_0\in
  X_+$ and a strictly positive functional $x_0^*$ such that
  $Sx_0=r(S_Y)x_0$, $S^*x_0^*=r(S_Y)x_0^*$, $Kx_0=r(K_Y)x_0$, and
  $K^*x_0^*=r(K_Y)x_0^*$. Now put $\lambda:=r(S_Y)$ and note that
  $r(K_Y)=r(K)$ by Proposition~\ref{sr-pR}. Also, observe that
  $r(S)\ge\lambda>0$ and $r(K)>0$ by Lemma~\ref{main-group}.

  It is left to show the ``furthermore'' clause. Fix $x>0$. It follows
  from Corollary~\ref{up-loc-r} that
  \begin{math}
    \liminf_n\norm{S^nx}^{\frac{1}{n}}\ge\lambda
  \end{math}
  and
  \begin{math}
    \liminf_n\norm{K^nx}^{\frac{1}{n}}\ge r(K).
  \end{math}
  However, we clearly have
  \begin{math}
    \limsup_n\norm{K^nx}^{\frac{1}{n}}\le r(K),
  \end{math}
  so that
  \begin{math}
    \lim_n\norm{K^nx}^{\frac{1}{n}}=r(K).
  \end{math}
\end{proof}

\begin{remark}\label{two-ii}
  \begin{enumerate}
  \item   It is easy to see that $\limsup_n\norm{T^nx}^{\frac{1}{n}}\le r(T)$
  for every operator $T$ and every vector $x$. Therefore, the conclusion   
  \begin{math}
    \lim_n\norm{K^nx}^{\frac{1}{n}}=r(K)
  \end{math}
  in the theorem is sharp.
  \item \label{dual-lim}
   Corollary~\ref{sr-loc-r} yields
  \begin{math}
    \liminf_n\norm{S^{*n}x^*}^{\frac{1}{n}}\ge\lambda
  \end{math}
  and
  \begin{math}
    \lim_n\norm{K^{*n}x^*}^{\frac{1}{n}}=r(K)
  \end{math}
  whenever $x^*>0$.
  \item   Clearly, the result (and the proof) remains valid if
  we require that $K$ is ideal irreducible instead of $S$.  Moreover,
  the result can be extended to any ideal irreducible commutative
  collection of operators containing a compact or a peripherally Riesz
  operator. In this case, the result will still be valid for every
  operator $S$ in the collection (with $\lambda$ depending on $S$).
  \end{enumerate}
\end{remark}

\section{Band irreducible semigroups}
\label{sec:bi}

In this section, we will show that most of the results of the
preceding sections remain valid if we replace ideal irreducibility
with band irreducibility under the additional assumption that all the
operators in $\iS$ are order continuous. This additional assumption is
justified by the following two facts. For $A\subseteq X$ we write
$I_A$ and $B_A$ for the ideal and the band generated by $A$,
respectively.  Suppose that $S$ is a positive order continuous
operator. If $S$ vanishes on a set $A\subseteq X_+$ then $S$ also
vanishes on $B_A$. Furthermore, if $J$ is an $S$-invariant ideal then
the band $B_J$ is still $S$-invariant.

\textbf{For the rest of this section, we will assume that $\iS$ is a
  semigroup of positive order continuous operators on a Banach lattice
  $X$.}  We start with a variant of Proposition~\ref{irr-chars} for
band irreducibility. Recall that for $x>0$ we write $B_{\iS x}$ for
the band generated by the orbit $\iS x$ of $x$ under $\iS$.

\begin{lemma}\label{bi-princ}
  $\iS$ is band irreducible iff $B_{\iS x}=X$ whenever $x>0$.
\end{lemma}

\begin{proof}
  Suppose that $\iS$ is band irreducible. It is easy to see that
  $B_{\iS x}$ is $\iS$-invariant for every $x>0$, so it suffices to
  prove that $\iS x\ne\{0\}$. For each $S\in\iS$, since $S$ is order
  continuous, its null ideal $N_S=\bigl\{x\in X\mid S\abs{x}=0\bigr\}$
  is a band. Therefore, $\bigcap_{S\in\iS}N_S$ is a band. It is easy
  to see that the intersection is $\iS$-invariant, hence it is
  zero. It follows that for every $x>0$ there exists $S\in\iS$ such
  that $Sx>0$, so that $\iS x$, and therefore $B_{\iS x}$, is
  non-zero.

  For the converse, suppose that $B$ is a non-zero proper
  $\iS$-invariant band. For each $0<x\in B$ we have $B_{\iS
    x}\subseteq B$, hence $B_{\iS x}\ne X$.
\end{proof}

\begin{proposition}\label{bi-chars}
  Suppose that $\iS$ is band irreducible. Then
  \begin{enumerate}
  \item\label{bi-ideals} every non-zero algebraic ideal in $\iS$ is
    band irreducible;
  \item\label{bi-loc} for any $x>0$ in $X$ and every order continuous
    $x^*>0$ in $X^*$ there exists $S\in\iS$ such that $\langle
    x^*,Sx\rangle\ne 0$;
  \item\label{bi-ops} $U\iS V\ne\{0\}$ for any non-zero $U,V\in
    L(X)_+$ provided that $U$ is order continuous.
  \end{enumerate}
\end{proposition}

\begin{proof}
  \eqref{bi-ideals} Let $\iJ$ be
  an algebraic ideal in $\iS$. Take any $x>0$. Then $y\in I_{\iJ x}$
  iff there exist $S_1,\dots,S_n\in\iJ$ and $\lambda\in\mathbb R_+$
  such that $\abs{y}\le\lambda(S_1+\dots+S_n)x$. In this case, for any
  $S\in\iS$ we have $\abs{Sy}\le\lambda(SS_1x+\dots+SS_n)x$, so that
  $Sy$ is in $I_{\iJ x}$. It follows that $I_{\iJ x}$ and, therefore,
  $B_{\iJ x}$ is $\iS$-invariant.

  Observe that $\iJ x$ and, therefore, $B_{\iJ x}$,
  is non-zero. Indeed, suppose that $\iJ x=\{0\}$ and fix any
  non-zero $T\in\iJ$. Then
  for every $S\in\iS$ we have $TS\in\iJ$ so that $TSx=0$. It follows
  that $T$ vanishes on $\iS x$ and, therefore, on $B_{\iS x}$. But
  $B_{\iS x}=X$ by Lemma~\ref{bi-princ}, so that $T=0$; a contradiction.

  Thus, the band $B_{\iJ x}$ is $\iS$-invariant and non-zero, hence
  $B_{\iJ x}=X$. Now Lemma~\ref{bi-princ} yields the required result.

  \eqref{bi-loc} Suppose not. Then $x^*$ vanishes on $\iS x$, hence on
  $B_{\iS x}$, so that $x^*=0$; a contradiction.

  \eqref{bi-ops} Suppose not, suppose $U\iS V=\{0\}$.  Since $V\ne 0$,
  there exists $x>0$ with $Vx>0$. Then $U$ vanishes on $\iS Vx$ and,
  therefore, on $B_{\iS Vx}$, so that, by Lemma~\ref{bi-princ}, $U=0$;
  a contradiction.
\end{proof}

Next, we use the idea of the proof of Lemma~3 of~\cite{Grobler:86} to extend
Theorem~\ref{Tur-pos} to the band irreducible case.

\begin{proposition}\label{Tur-pos-bi}
  If all the operators in $\iS$ are compact and quasi-nilpotent then
  $\iS$ is band reducible. 
\end{proposition}

\begin{proof}
  Let $F$ be the closed ideal generated by the union of the ranges of
  all the operators in $\iS$. We may assume, without loss of
  generality, that $\dim F>1$ as, otherwise, $F$ is a
  band and we are done. Applying Theorem~\ref{Tur-pos} to the
  restriction of $\iS$ to $F$, we find a non-zero closed
  ideal $J\subsetneq F$ such that $J$ is $\iS$-invariant. It follows that
  $B_J$ is $\iS$-invariant. It is left to show that $B_J$ is
  proper. Suppose that $B_J=X$. Then for any $x\in X_+$ we have
  $x_\alpha\uparrow x$ for some net $(x_\alpha)$ in $J_+$. Let
  $S\in\iS$. Since $S$ is order continuous, we have $Sx_\alpha\uparrow
  Sx$. Since $S$ is compact, after passing to a subnet we know
  that $(Sx_\alpha)$ converges in norm; hence $Sx_\alpha\to Sx$ in
  norm. It follows that $Sx\in J$. Since $x>0$ was arbitrary, it
  follows that $F\subseteq J$; a contradiction.
\end{proof}

\begin{numbered}\label{oc-list}
  One can now easily verify that the results of the previous sections
  remain true for band irreducible semigroups of order continuous
  operators with the following straightforward modifications.
\begin{itemize}
\item In Corollary~\ref{faith}, one has to assume that $x^*$ is
  \emph{order continuous}.
\item In~\ref{fin-rank-proj}, we now only consider order continuous
  projections. It is easy to see that the functionals $x^*_1,\dots,x^*_n$
  defined there are also order continuous.
\item Proposition~\ref{S_P} extends as long as $P$ is order
  continuous.
\item In Corollary~\ref{ideal-dense}, we now conclude that $Y$ is a
  closed $\iS$-invariant sublattice of $X$ and the band generated by
  $Y$ is all of $X$.
\item In Corollary~\ref{sr-eigen}, we replace ``quasi-interior'' with
  ``a weak unit''.
\item In~\ref{block-matrix}, we replace $\overline{I_{x_i}}$ with
  $B_{x_i}$. 

\item In Corollary~\ref{sr-loc-r} we need to assume that $x^*$ is
  $\sigma$-order continuous, because in this case we still have
  $x^*(x_0)>0$ (recall that $x_0$ is now a weak unit). In particular,
  $S^*$ is strictly positive on $\sigma$-order continuous functionals.
\item In Corollary~\ref{dual-eigen}, the functional $x_0^*$ is now
  order continuous because $x_0^*=x_1^*+\dots+x_r^*$ and
  $x_1^*,\dots,x_r^*$ are order continuous.
\item Proposition~\ref{r-sublat} remains valid for $S$. For $S^*$ we
  can only say that if there is an $\sigma$-order continuous
  eigenfunctional $x^*$ for $r(S)$ then $r(S_Y)=r(S)$ and $\abs{x^*}$
  is also in the eigenspace. Indeed, as in the proof of
  Proposition~\ref{r-sublat}, we get
  \begin{displaymath}
    r(S_Y)\abs{x^*}\le r(S)\abs{x^*}\le S^*\abs{x^*}
    \quad\mbox{and}\quad
    \bigl(S^*\abs{x^*}-r(S_Y)\abs{x^*}\bigr)(x_0)=0.
  \end{displaymath}
  Since $x^*$ is $\sigma$-order continuous, so are $\abs{x^*}$ and
  $S^*\abs{x^*}$ (because $S$ is order continuous and
  $S^*\abs{x^*}=\abs{x^*}\circ S$). It follows that
  \begin{math}
    S^*\abs{x^*}-r(S_Y)\abs{x^*}
  \end{math}
  is a $\sigma$-order continuous functional in $X_+$ vanishing on a
  weak unit $x_0$, hence $S^*\abs{x^*}=r(S_Y)\abs{x^*}$. Therefore,
  $r(S_Y)=r(S)$ and $\abs{x^*}$ is in the eigenspace.
\end{itemize}
\end{numbered}

Next, we consider finitely generated semigroups. The difficulty here
is that in order to use our previous results, we need $\Rplus T$ to
consist of order continuous operators. However, we do not know whether
this follows from the assumption that $T$ itself is order continuous
(cf. the counterexample in Section~3 of~\cite{Kitover:05a}).

\begin{lemma}\label{soc-bi-str-pos}
  Let $S$ and $T$ be two commuting non-zero positive $\sigma$-order
  continuous operators. If $T$ is band irreducible then $S$ is
  strictly positive.
\end{lemma}

\begin{proof}
  Suppose not, suppose $Sx=0$ for some $x>0$. Without loss of
  generality, $\norm{T}<1$, so that $z:=\sum_{n=0}^\infty T^nx$
  exists. Clearly, $Tz\le z$. It follows that $B_z$ in invariant under
  $T$ and, therefore, $B_z=X$. On the other hand, we have $Sz=0$, so
  that $S$ vanishes on $B_z$, so $S=0$; a contradiction.
\end{proof}

There have been several variants of the Perron-Frobenius Theorem for
band irreducible operators, see
e.g.,~\cite{Grobler:86,Grobler:87,Grobler:95,Kitover:05}.  The
following variant was proved in Theorems~5.2 and~5.3
in~\cite{Grobler:95}.

\begin{theorem}[Grobler]\label{Grobler}
  If $T$ is positive, band irreducible, $\sigma$-order continuous, and
  power compact then $r(T)>0$, $\sigma_{\rm per}(T)=r(T)G$ where $G$
  is the set of all $m$-th roots of unity for some $m\in\mathbb N$,
  and each point in $\sigma_{\rm per}(T)$ is a simple pole of the
  resolvent with one-dimensional eigenspace.
\end{theorem}

We can now easily deduce this result (and more) from our
techniques. Namely, we claim that $T$ enjoys the conclusion of
Theorem~\ref{single}. In particular, the peripheral spectral subspace
of $T$ is spanned by disjoint positive vectors and $T$ acts as a
scalar multiple of a cyclic permutation on these vectors. This easily
implies the conclusion of Theorem~\ref{Grobler}.

Indeed, suppose that $T^m$ is compact. Then $T$ and, therefore, $T^m$
is strictly positive by Lemma~\ref{soc-bi-str-pos}. By Lemma~9.30
of~\cite{Abramovich:02}, all the operators in $\Rplus T$ are order
continuous. Then the results of Section~\ref{sec:sr} and~\ref{sec:up}
apply to $\Rplus T$ (again, the proofs must be adjusted as
in~\ref{oc-list}). In particular, $\Rplus T$ contains no
quasinilpotent operators, so that $r(T)>0$. For simplicity, we can
scale $T$ so that $r(T)=1$. Now, the proof of Theorem~\ref{single}
remains valid for $T$.

Next, we extend this result beyond power compact operators.

\begin{lemma}\label{RT-so-norm-cont}
  Suppose that $T\in L(X)_+$ and some power of $T$ is
  $\sigma$-order-to-norm continuous. Then every operator in the
  asymptotic part of $\Rplus T$ is $\sigma$-order-to-norm continuous.
\end{lemma}

\begin{proof}
  Suppose that $T^m$ is $\sigma$-order-to-norm
  continuous and $S=\lim_jb_jT^{n_j}$.  Suppose that $x_k\downarrow
  0$. Fix a positive real $\varepsilon$. Fix $j$ such that
  $n_j\ge m$ and $\norm{S-b_jT^{n_j}}<\varepsilon$. Observe that
  \begin{displaymath}
    \norm{Sx_k}\le\bignorm{S-b_jT^{n_j}}\norm{x_k}+\bignorm{b_jT^{n_j}x_k}
  \end{displaymath}
  Note that
  $\bignorm{S-b_jT^{n_j}}\norm{x_k}\le\varepsilon\norm{x_1}$. On the
  other hand, $x_k\downarrow 0$ yields $T^mx_k\to 0$, so that 
  \begin{math}
    b_jT^{n_j}x_k=\bigl(b_jT^{n_j-m}\bigr)T^mx_k\to 0
  \end{math}
  in norm as $k\to\infty$. It follows that $Sx_k\to 0$.
\end{proof}

\begin{corollary}\label{RT-oc}
  Suppose that $T\in L(X)_+$ is peripherally Riesz, band irreducible,
  and $\sigma$-order continuous. If some power of $T$ is
  $\sigma$-order-to-norm continuous then every operator in $\Rplus T$
  is order continuous.
\end{corollary}

\begin{proof}
  By Lemma~\ref{RT-so-norm-cont}, every operator in $\Rplus T$ is
  $\sigma$-order continuous. It follows from
  Proposition~\ref{dych-proj-nilp} that $\Rplus T$ contains a non-zero
  compact operator; denote it by $K$. By Lemma~\ref{soc-bi-str-pos},
  $K$ is strictly positive. The result now follows from Corollary~9.16
  of~\cite{Abramovich:02}.
\end{proof}

In particular, if $T$ is peripherally Riesz, band irreducible, and
$\sigma$-order-to-norm continuous with $r(T)=1$ then $\Rplus T$
consists of order continuous operators and, in view of the preceding
remarks, all the conclusions of Theorem~\ref{single} remain valid. The
proof is analogous. Note that this fact is a generalization
of Theorem~\ref{Grobler} because a compact positive $\sigma$-order
continuous operator is automatically $\sigma$-order-to-norm
continuous.

\medskip

Theorem~\ref{Grobler} can be extended from power compact to power
strictly singular operators%
\footnote{Note that if $T^m$ is strictly singular for some $m$ then
  $r_{\rm ess}(T)^m=r_{\rm ess}(T^m)=0$. Hence, every
  non-quasi\-nilpotent power strictly singular operator is
  peripherally Riesz.}.  Suppose that $T$ is strictly singular. It
follows from Corollary~3.4.5 on p.~193 of~\cite{Meyer-Nieberg:91} that
$T$ is order weakly compact, i.e., it takes order intervals into
relatively weakly compact sets. Suppose that, in addition, $T$ is
$\sigma$-order continuous. It is now easy to see that $T$ is
$\sigma$-order-to-norm continuous. Indeed, suppose that $x_n\downarrow
0$. Then $Tx_n\downarrow 0$ and, by Eberlein-\v Smulian Theorem there
exists a subsequence $(x_{n_k})$ such that $Tx_{n_k}$ converges
weakly. Since $(Tx_{n_k})$ is monotone, it converges in norm by
Theorem~3.52 of~\cite{Aliprantis:06}. It follows that $Tx_{n_k}\to 0$,
so that $Tx_n\to 0$. Now Corollary~\ref{RT-oc} yields the following
result.

\begin{corollary}
  Suppose that $T\ge 0$ is $\sigma$-order continuous, band
  irreducible, and power strictly singular, and $r(T)=1$.
  Then all the conclusions of Theorem~\ref{single} are valid for $T$.
\end{corollary}

Finally, we can also extend Theorem~\ref{two-op-ii} as follows (it can
also be viewed as an extension of Corollary~9.34 in~\cite{Abramovich:02}).

\begin{theorem}
  Suppose that $S$ and $K$ are two non-zero commuting positive
  operators such that $K$ is compact, $\sigma$-order continuous and
  band irreducible. Then there exists a weak unit $x_0\in X_+$, a
  strictly positive functional $x_0^*$, and a positive real $\lambda$
  such that $Sx_0=\lambda x_0$, $S^*x_0^*=\lambda x_0^*$,
  $Kx_0=r(K)x_0$, and $K^*x_0^*=r(K)x_0^*$. Furthermore,
  \begin{math}
    \liminf_n\norm{S^nx}^{\frac{1}{n}}\ge\lambda
  \end{math}
  and
  \begin{math}
    \lim_n\norm{K^nx}^{\frac{1}{n}}=r(K)
  \end{math}
  whenever $x>0$.
\end{theorem}

\begin{proof}
  Let $\iS=\Rplus\{S,K\}$. Then $\iS$ is commutative and band
  irreducible. By Lemma~\ref{soc-bi-str-pos}, $K$ is strictly
  positive. It follows from Lemma~9.30 of~\cite{Abramovich:02} that
  all the operators in $\iS$ are order continuous. Hence, all the
  results of Sections~\ref{sec:sr} and~\ref{sec:up} apply with the
  modifications described in~\ref{oc-list}. The rest of the proof is
  exactly as in Theorem~\ref{two-op-ii} with the only exception that,
  instead of being quasi-interior, $x_0$ is now a weak unit.
\end{proof}

\begin{remark}
  Using Corollary~\ref{sr-loc-r}, which remains valid for
  band-irreducible semigroups as long as $x^*$ is $\sigma$-order
  continuous, we can show, as in Remark~\ref{two-ii}\eqref{dual-lim},
  that
  \begin{math}
    \liminf_n\norm{S^{*n}x^*}^{\frac{1}{n}}\ge\lambda
  \end{math}
  and
  \begin{math}
    \lim_n\norm{K^{*n}x^*}^{\frac{1}{n}}=r(K)
  \end{math}
  whenever $x^*>0$ is $\sigma$-order continuous.
\end{remark}

\begin{remark}
  As in Theorem~\ref{two-op-ii}, the result can be extended to any
  commutative semigroup of $\sigma$-order continuous operators
  containing a band irreducible operator and a non-zero compact
  operator. Indeed, by Lemma~\ref{soc-bi-str-pos}, the compact
  operator is strictly positive, so that all the operators in the
  semigroup are order continuous
  by~\cite[Lemma~9.30]{Abramovich:02}. Now we can apply results of
  Sections~\ref{sec:sr} and~\ref{sec:up}.
\end{remark}

\section{One-sided ideals of $\iS$}
\label{sec:ideals}

Some of the properties of an irreducible semigroup can be
characterized in terms of minimal right ideals of $\iS$. Throughout
this section, we assume that $\iS$ is an $\Rplus$-closed ideal
irreducible semigroup of positive operators on a Banach lattice $X$
with $r=\minrank\iS<+\infty$. We write $\iP_r$ for the set of all
projections of rank $r$ in $\iS$.

\begin{lemma}\label{rideal-minproj}
  Every non-zero (right or left) ideal in $\iS$ contains a projection of rank
  $r$.
\end{lemma}

\begin{proof}
  Let $\iJ$ be a right ideal in $\iS$. Take any $0\ne T\in\iJ$. Since
  $\iS_r$ is ideal irreducible by Remark~\ref{S_r}, $T\iS_r\ne\{0\}$
  by Proposition~\ref{irr-chars}\eqref{irr-ops}. Replacing $T$ with a
  non-zero operator in $T\iS_r$ we may assume without loss of
  generality that $\rank T=r$. By Lemma~\ref{rAS1}, there exists
  $A\in\iS$ with $r(TA)=1$; replacing $T$ with $TA$ we may assume that
  $r(T)=1$. Let $P$ be the spectral projection of $T$ for $\sigma_{\rm
    per}(T)$, then $P\in\iS$ by Remark~\ref{dych-no-nilp}. It follows
  that $\rank P=\rank T=r$ and, therefore, $PT=T$. Also, by
  Proposition~\ref{S_P-perm}, $PTP$ is invertible in the sense that
  there exists $S\in\iS$ such that $(PTP)(PSP)=P$. It follows that
  $TPSP=P$, so that $P\in\iJ$. The proof for a left ideal is similar
  because $r(TA)=r(AT)$ and $T$ commutes with $P$.
\end{proof}

\begin{corollary}\label{rideals}
  Minimal right ideals in $\iS$ are exactly of form $P\iS$, where $P\in\iP_r$.
  In this case, $P\iS$ is an ideal iff $\Range P$ is $\iS$-invariant.
\end{corollary}

\begin{proof}
  Suppose $\iJ$ is a minimal right ideal. By
  Lemma~\ref{rideal-minproj}, there is a projection $P$ in $\iJ$ with
  $\rank P=r$.  Since $P\iS$ is a right ideal, by minimality we have
  $P\iS=\iJ$.

  Conversely, suppose that $P\in\iP_r$; show that $P\iS$ is a minimal
  right ideal. Suppose that $\iJ$ is a non-zero right ideal in $\iS$ and
  $\iJ\subseteq P\iS$. Again, by Lemma~\ref{rideal-minproj}, there
  exists a projection $Q\in\iJ$ such that $\rank Q=r$. It follows from
  $Q\in P\iS$ that $\Range Q=\Range P$. Therefore, $QP=P$, so that
  $P\in\iJ$.  Hence, $P\iS=\iJ$.

  If $P\iS$ is an ideal then $SP=SP^2\in SP\iS\subseteq P\iS$ for
  every $S\in\iS$, so that $S(\Range P)=\Range SP\subseteq\Range
  P$. Conversely, if $\Range P$ is $\iS$-invariant then for any
  $S,T\in\iS$ we have $\Range TPS\subseteq\Range P$, so that
  $TPS=PTPS\in P\iS$; hence $P\iS$ is an ideal.
\end{proof}

The next fact can be viewed as an extension of Proposition~\ref{sr}.

\begin{proposition}\label{sr-rid}
  The following are equivalent.
  \begin{enumerate}
  \item\label{sr-rid-proj} All projections in $\iP_r$ have the same range;
  \item\label{sr-all-rid} All minimal right ideal in $\iS$ are ideals;
  \item\label{sr-some-rid} Some minimal right ideal in $\iS$ is an ideal;
  \item\label{sr-uniq-id} $\iS$ has a unique minimal right ideal.
  \end{enumerate}
\end{proposition}

\begin{proof}
  \eqref{sr-rid-proj}$\Rightarrow$\eqref{sr-all-rid} follows from
  Proposition~\ref{sr} and Corollary~\ref{rideals}.

  \eqref{sr-all-rid}$\Rightarrow$\eqref{sr-some-rid} is trivial.

  \eqref{sr-some-rid}$\Rightarrow$\eqref{sr-uniq-id} Suppose that
  $P\iS$ is an ideal for some $P\in\iP_r$. Let $Q\in\iP_r$. Since
  $Q\iS P\ne\{0\}$ by Proposition~\ref{irr-chars}\eqref{irr-ops}, we
  have $QSP\ne 0$ for some $S\in\iS$. Note that $QSP\iS$ is a right
  ideal and $QSP\iS\subseteq Q\iS$, it follows from minimality that
  $QSP\iS=Q\iS$. On the other hand, since $P\iS$ is an ideal,
  $QSP\iS\subseteq P\iS$, so that $Q\iS\subseteq P\iS$. Again, by
  minimality, we have $Q\iS=P\iS$. 

  \eqref{sr-uniq-id}$\Rightarrow$\eqref{sr-rid-proj} Let
  $P,Q\in\iP_r$. Then $P=P^2\in P\iS=Q\iS$, hence
  $\Range P\subseteq\Range Q$. Similarly, $\Range Q\subseteq\Range P$.
\end{proof}

Corollary~\ref{rideals} and Proposition~\ref{sr-rid} extend
Lemmas~5.2.3, 5.2.4(i,ix-xi), and~8.7.18 in~\cite{Radjavi:00}.

\providecommand{\bysame}{\leavevmode\hbox to3em{\hrulefill}\thinspace}
\providecommand{\MR}{\relax\ifhmode\unskip\space\fi MR }

\end{document}